\documentclass{AIMS}
\usepackage{amsmath}
  \usepackage{paralist}
  \usepackage{graphics} 
  \usepackage{epsfig} 
\usepackage{graphicx}  \usepackage{epstopdf}
 \usepackage[colorlinks=true]{hyperref}
\hypersetup{urlcolor=blue, citecolor=red}

\def\currentvolume{X}
 \def\currentissue{X}
  \def\currentyear{200X}
   \def\currentmonth{XX}
    \def\ppages{X--XX}
     \def\DOI{10.3934/xx.xx.xx.xx}

\newtheorem{theorem}{Theorem}[section]
\newtheorem{corollary}[theorem]{Corollary}

\newtheorem{example}[theorem]{Example}  
\newtheorem{lemma}[theorem]{Lemma}
\newtheorem{proposition}[theorem]{Proposition}

\theoremstyle{definition}

\newtheorem{remark}[theorem]{Remark}

\newcommand{\PP}{\mathcal{P}}

\newcommand{\si}{\sigma}
\def\deg{{\rm deg\,}}

\newcommand{\be}{\begin{equation}}
\newcommand{\ee}{\end{equation}}

\newcommand{\NN}{\mathbb{N}}

\newcommand{\e}{\varepsilon}
\newcommand{\al}{\alpha}
\newcommand{\de}{\delta}

\def\AA{{\cal A}}

\def\si{\sigma}

\title[Lower radius and spectral mapping theorem] 
      {Lower spectral radius and spectral mapping theorem for suprema preserving mappings}

\author[Vladimir M\"uller and Aljo\v{s}a Peperko]{}


 \email{muller@math.cas.cz}
 \email{aljosa.peperko@fmf.uni-lj.si}
 \email{aljosa.peperko@fs.uni-lj.si}

\begin{document}
\maketitle

\centerline{\scshape Vladimir M\"uller}
\medskip
{\footnotesize
 \centerline{Institute of Mathematics, Czech Academy of Sciences}
   \centerline{\v{Z}itna 25 }
   \centerline{115 67 Prague, Czech Republic}
} 

\medskip

\centerline{\scshape Aljo\v{s}a Peperko$^*$}
\medskip
{\footnotesize
 \centerline{Faculty of Mechanical Engineering, University of Ljubljana }
   \centerline{A\v{s}ker\v{c}eva 6}
   \centerline{SI-1000 Ljubljana, Slovenia}
\centerline{and}
 \centerline{Institute of Mathematics, Physics and Mechanics  }
   \centerline{Jadranska 19}
   \centerline{SI-1000 Ljubljana, Slovenia}
}

\bigskip


\begin{abstract}
We study  Lipschitz, positively homogeneous and finite suprema preserving mappings defined on a max-cone of positive elements in a normed vector lattice.  
We prove that the lower spectral radius of such a mapping is always a minimum value of its approximate point spectrum.
We apply this result to show that   the spectral mapping theorem holds for  the approximate point spectrum of such a mapping. 
By applying this spectral mapping theorem we obtain new inequalites for the Bonsall cone spectral radius of max type kernel operators. 
\end{abstract}

\section{Introduction}
Max-type operators (and corresponding max-plus type operators and their tropical versions known also as Bellman operators) arise in a large field of problems from the theory of differential and difference equations,  mathematical physics, optimal control problems, discrete mathematics, turnpike theory, mathematical economics, mathematical biology,  games and controlled Markov processes, generalized solutions of the Hamilton-Jacobi-Bellman differential equations, continuously observed and controlled quantum systems, discrete and continuous dynamical systems, ... (see e.g. \cite{MN02}, \cite{KM97},  
\cite{LM05}, \cite{LMS01},  
\cite{AGN}, \cite{MP17} and the references cited there). The eigenproblem of such operators has so far received substantial attention due to its applicability in the above mentioned problems (see e.g. \cite{MN02}, \cite{KM97}, \cite{AGN}, \cite{AGW04}, \cite{MP17}, \cite{LN12}, 
\cite{MN10}, 
\cite{B98}, \cite{MP15}, 
\cite{MP12}, 
\cite{AGH15}, \cite{S07} and the references cited there). However, there seems to be a lack of general treatment of spectral theory for such operators, even though the spectral theory for nonlinear operators on Banach spaces is already quite well developed
(see e.g. \cite{APV04}, \cite{APV00}, \cite{AGV02}, \cite{F97}, 
\cite{SV00}, \cite{MP17} and the references cited there). One of the reasons for this might lie in the fact that these operators behave nicely on a suitable subcone (or subsemimodule), but less nicely on the whole (Banach) space. Therefore it appears, that it is not trivial to directly apply this known non-linear spectral theory to obtain satisfactory information on a restriction to a given cone of a max-type operator.   
 The Bonsall cone spectral radius  plays the role of the spectral radius in this theory (see e.g.  \cite{MN02}, \cite{MN10}, \cite{MP17}, \cite{AGN}, \cite{LN11}, \cite{Gr15}, \cite{MP15} and the references cited there). 

In \cite{MP17}, we studied Lipschitz, positively homogeneous and finite suprema preserving mappings defined on a max-cone of positive elements in a normed vector lattice. We showed that for such a mapping the Bonsall cone spectral radius is the maximum value of its
 approximate point spectrum (see Theorem \ref{appl_to_matrices} below). We also proved that an analogue of this  result holds  also for  Lipschitz, positively homogeneous and additive mappings defined on a normal convex cone in a normed space (see Theorem \ref{main_normal} below).  

The current article may be considered as a continuation of \cite{MP17}. It is organized as follows. In Section 2 we recall some definitions and results that we will use in the sequel. We show in Section 3 that the lower spectral radius of a mapping from both of the above decribed settings is   a minimum value of its approximate point spectrum (Theorems \ref{dAisMin} and \ref{dAnormal}). In Section 4 we apply this result to show that the maxpolynomial spectral mapping theorem  holds for the approximate point spectrum of Lipschitz, positively homogeneous and finite suprema preserving mappings (Theorem \ref{main_spectral_mapping}). In the last section we use this spectral mapping theorem to prove some new inequalities  for the Bonsall cone spectral radius of Hadamard products of max type kernel operators (Theorem \ref{powerineq}).

\section{Preliminaries}

A subset $C$ of a real vector space $X$ is called a cone (with vertex 0) if 
$tC \subset C$ for all $t \ge 0$, where $tC =\{tx : x \in C \} $. A map $A: C \to C$ is called positively homogeneous
(of degree 1) if $A(tx) = tA(x)$ for all $t \ge 0$ and $x \in C$. We say that the cone
$C$ is pointed if $C \cap (- C) = \{0\}$. 

A convex pointed cone $C$ of $X$ induces on $X$ a
partial ordering $\le$, which is defined by $x  \le y$ if and only if  $y - x \in C$. In this case $C$ is denoted by $X_+$ and $X$ is called an ordered vector space. If, in addition, $X$ is a normed space then it is called an ordered normed space.  If, in addition, the norm is complete, then  $X$ is called an ordered Banach space.

A convex cone $C$ of $X$ is  called a wedge. A wedge induces on $X$  (by the above relation) a vector preordering $\le$ (which is reflexive, transitive, but not necessary antisymmetric).
We say that the cone $C$ is proper if it is closed, convex and pointed. A cone $C$ of a normed space $X$  is called normal if there exists a constant $M$ such
that $\|x\| \le M \|y\|$ whenever $ x \le y$, $x,y\in C$. A convex and pointed  cone $C=X_+$ of an ordered normed space $X$ is normal if and only if there exists an equivalent monotone norm $||| \cdot|||$ on $X$, i.e.,  
$|||x||| \le |||y|||$ whenever $0 \le x \le y$ (see e.g. \cite[Theorem 2.38]{AT07}).
Every proper cone $C$ in a finite dimensional
Banach space  is necessarily normal.

If $X$ is a normed linear space, then a cone $C$ in $X$ is said to be complete if it is a complete metric space in the topology induced by $X$. In the case when $X$ is a Banach space this is equivalent to $C$ being closed in $X$.

If $X$ is an ordered vector space, then a cone $C\subset X_+$ is called a max-cone if for every pair  $x,y\in C$ there exists a supremum $x\vee y$  (least upper bound) in $C$. 
We consider here on $C$ an order inherited from $X_+$.
A map $A:C\to C$ preserves finite suprema on $C$ if $A(x\vee y)=Ax\vee Ay\quad(x,y\in C$). If $A:C\to C$ preserves finite suprema, then it is monotone (order preserving) on $C$, i.e., $Ax \le Ay$ whenever $x\le y$, $x,y \in C$  .

An ordered vector space  $X$ is called a vector lattice (or a Riesz space) if every two vectors $x,y \in X$ have a supremum and  infimum (greatest lower bound) in $X$. A positive cone $X_+$ of a vector lattice $X$ is called a lattice cone.

Note that by \cite[Corollary 1.18]{AT07} a pointed convex cone $C=X_+$ of an ordered vector space $X$ is the lattice cone for the vector subspace $C-C$ generated by $C$ in $X$, if and only if  $C$ is a max cone (in this case a supremum of $x, y\in C$ exists in $C$ if only if it exists in $X$; and suprema coincide). 

If $X$ is a vector lattice, then the absolute value of $ x\in X$ is defined by $|x|= x \vee (-x)$. A vector lattice and  normed vector space is called a normed vector lattice (a normed Riesz space) if $|x| \le |y|$ implies $\|x\| \le \|y\|$. A complete normed vector lattice is called a Banach lattice. A positive cone $X_+$ of a normed vector lattice $X$ is  proper and normal. 

In a vector lattice $X$  the following Birkhoff's inequality for $x_1,\dots,x_n,y_1,\dots,y_n\in X$ holds:
\be
\Bigl|\bigvee_{j=1}^n x_j-\bigvee_{j=1}^n y_j\Bigr| \le \sum_{j=1}^n |x_j-y_j\Bigr|.
\label{Birk_inq}
\ee
For the theory of  vector lattices, Banach lattices, cones, wedges, operators on cones and applications e.g. in financial mathematics we refer the reader 
to 
 \cite{AA02}, \cite{AT07},
 \cite{AB85}, \cite{W99}, \cite{ABB90}, \cite{LT96}, \cite{JM14}, \cite{MP17}, \cite{LN08}, \cite{AB06}  
 and the references cited there. 

Let $X$ be a normed space and $C \subset X$ a non-zero cone. Let $A:C \to C$ be positively homogeneous and bounded, i.e., 
$$\|A\|:=\sup \left \{\frac{\|Ax\|}{\|x\|} :x\in C, x\neq 0\right \}<\infty.$$
It is easy to see that
$\|A\|=\sup\{\|Ax\|:x\in C, \|x\|\le 1 \}$ and $\|A^{m+n}\|\le \|A^m\|\cdot \|A^n\|$ for all $m,n\in\NN$. It is well known that this implies that the limit $\lim_{n\to\infty}\|A^n\|^{1/n}$ exists and is equal to $\inf_n \|A^n\|^{1/n}$. The limit $r(A):=\lim_{n\to\infty}\|A^n\|^{1/n}$ is called the Bonsall cone spectral radius of $A$. The approximate point spectrum $\sigma_{ap}(A)$ of $A$ is defined as the set of all $s\ge 0$ such that $\inf\{\|Ax-sx\|:x\in C,\|x\|=1\}=0$. The (distinguished) point spectrum $\si_p(A)$ of 
$A$ is defined by
$$
\si_p(A)=\Bigl\{s\ge 0:\hbox{ there exists } x\in C, x\ne 0\hbox{ with } Ax=sx\Bigr\}.
$$

For $x\in C$ define the local cone spectral radius by $r_x(A):=\limsup_{n\to\infty}\|A^nx\|^{1/n}$. Clearly $r_x(A)\le r(A)$ for all $x\in C$. It is known that the equality 
\be 
\sup\{r_x(A):x\in C\}=r(A)
\label{eq}
\ee
is not valid in general. In \cite{MN02} there is an example of a proper cone $C$ in a Banach space $X$ and  a positively homogeneous and continuous (hence bounded)  map $A : C \to C$ such that $\sup\{r_x(A):x\in C\}< r(A)$. A recent example of such kind, where  $A$ is in addition monotone, is obtained in \cite[Example 3.1]{Gr15}. However, if $C$ is a normal, complete, convex and pointed cone in a normed space $X$  and  $A : C \to C$ is  positively homogeneous, monotone  and continuous, then 
\cite[Theorem 3.3]{MN10}, 
 \cite[Theorem 2.2]{MN02} and \cite[Theorem 2.1]{Gr15} ensure that (\ref{eq}) is valid.

If $X$ is a Banach lattice,  $C\subset X_+$  a max-cone and $A:C\to C$ a mapping which is bounded, positively homogeneous and preserves finite suprema, then the equality (\ref{eq}) is not necessary valid as shown in \cite{MP17}. Some additional examples of maps for which (\ref{eq}) is not valid can be found in \cite{Gr15}.

Let $C$ be a cone in a normed space $X$ and $A:C \to C$. Then $A$ is called  Lipschitz if there exists $L >0$ such that $\|Ax-Ay \|\le L\|x-y\| $ for all $x,y \in C$. 

The following two results were the main results of \cite{MP17} (see \cite[Theorem 3.6 and Corollaries 1 and 2]{MP17} and  \cite[Theorem 4.2 and Corollary 4]{MP17}).

\begin{theorem} Let $X$ be a normed vector lattice, let $C\subset X_+$ be a non-zero max-cone. Let $A:C\to C$ be a mapping which is bounded, positively homogeneous and preserves finite suprema.
Let $C'\subset C$ be a bounded subset satisfying
$\|A^n\|=\sup\{\|A^nx\|: x\in C'\}$ for all $n$.
Then
$$
\bigl[\sup\{r_x(A):x\in C'\} , r(A)\bigr]  \subset \sigma_{ap}(A).
$$
In particular, $r(A)\in\sigma_{ap}(A)$.
Moreover, $r_x(A)\in\sigma_{ap}(A)$ for each $x\in C$.

If, in addition, $A$ is a Lipschitz, then $r(A)= \max \bigl\{t : t \in \sigma_{ap}(A)\bigr\}$.
\label{appl_to_matrices}
\end{theorem}

\begin{theorem} 
Let $X$  be a normed space, $C\subset X$ a non-zero normal wedge and let $A:C\to C$ be positively homogeneous, additive and Lipschitz.
Let $C'\subset C$ be a bounded subset satisfying
$\|A^n\|=\sup\{\|A^nx\|: x\in C'\}$ for all $n$.
Then
$$
\bigl[\sup\{r_x(A):x\in C'\} , r(A) \bigr]  \subset \sigma_{ap}(A).
$$
In particular, $r(A)= \max \bigl\{t : t \in \sigma_{ap}(A)\bigr\}$.

 Moreover, $r_x(A)\in\sigma_{ap}(A)$ for each $x\in C$, $x\ne 0$.
\label{main_normal}
\end{theorem}

\section{Lower spectral radius and approximate point spectrum}

Let $X$ be a normed space and $C \subset X$ a non-zero cone. Let $A:C \to C$ be positively homogeneous and bounded mapping.
 Let $m(A):=\inf\{\|Ax\|:x\in C, \|x\|=1\}$ be the minimum modulus of $A$ (note that in \cite{APV04}, $m(A)$ is called the inner norm of $A$ in the case when $C=X$). Observe  that $m(A)=0$ if and only if $0 \in  \sigma_{ap}(A)$. It is easy to see that $m(A^{n+m})\ge m(A^n)m(A^m)$ for all $m,n\in\NN$. It is well known that this implies that the limit $\lim_{n\to\infty}m(A^n)^{1/n}$ exists and is equal to the supremum $\sup_n m(A^n)^{1/n}$.
Let 
$$
d(A)=\lim_{n\to\infty}m(A^n)^{1/n}
$$
be the "lower spectral radius" of $A$.

If $A:C\to C$ is bijective and $A^{-1}$ is bounded, then $m(A)=\|A^{-1}\|^{-1}$ and
$d(A)=r(A^{-1})^{-1}$. If $X$, $C$ and $A^{-1}$ satisfy the assumptions of Theorem \ref{appl_to_matrices} or of Theorem \ref{main_normal}, then $r(A^{-1})\in\si_{ap}(A^{-1})$ and $\si_{ap}(A)=\{t^{-1}:t\in\si_{ap}(A^{-1})\}$ and thus also $d(A)\in\si_{ap}(A)$. We show below in Theorems \ref{dAisMin} and \ref{dAnormal} that this is true, not only for invertible mappings $A$, but also under similar assumptions as in Theorems \ref{appl_to_matrices} and \ref{main_normal}.

First we observe the following result.
\begin{proposition}
Let $X$ be a normed space and $C \subset X$ a non-zero cone. If $A:C \to C$ is positively homogeneous and bounded, then
\be
d(A)\le\inf\bigl\{r_x(A): x\in C, x\ne 0\bigr\}\le
\sup\bigl\{r_x(A): x\in C, x\ne 0\bigr\}\le r(A).
\label{d_ineq_r}
\ee
If, in addition, $A$ is Lipschitz, then $\si_{ap}(A)\subset [d(A),r(A)]$.
\label{basic}
\end{proposition}
\begin{proof} If $x\in C$, $\|x\|=1$ then $m(A^j)\le \|A^jx\|\le\|A^j\|$ for all $j\in\NN$. So $d(A)\le r_x(A)\le r(A)$, which proves (\ref{d_ineq_r}).

If, in addition, $A$ is Lipschitz and $t\in\si_{ap}(A)$ then one can see easily that $m(A)\le t\le\|A\|$ 
 and $t^n\in\si_{ap}(A^n)$ for all $n\in\NN$ (see also the proof of \cite[Lemma 3.3]{MP17}). 
Since $m(A^n)\le t^n\le \|A^n\|$, it follows that $d(A)\le t\le r(A)$, which completes the proof.
\end{proof}

The following example shows that $\si_{ap}(A)$ may not contain the whole interval $\bigl[d(A), \inf\{r_x(A): x\in C,\; x\ne 0\}\bigr]$.
\begin{example} {\rm
Let $X=\ell^\infty$ with the standard basis $e_{n,k}\quad(n,k\in\NN)$. More precisely, the elements of $X$ are formal sums $x=\sum_{n,k\in\NN}\alpha_{n,k}e_{n,k}$ with real coefficient $\alpha_{n,k}$ such that
$$
\|x\|:=\sup\{|\alpha_{n,k}|: n,k\in\NN\} <\infty.
$$
Then $X$ is  a Banach lattice with the natural order.
Let $C=X_+$ and let $A:C\to C$ be defined by
$Ae_{n,1}=n^{-1}e_{n,2}$, $Ae_{n,k}=e_{n,k+1}\quad(k\ge 2).$ More precisely, 
$$
A\Bigl(\sum_{n,k\in\NN}\alpha_{n,k}e_{n,k}\Bigr)=
\sum_{n\in\NN} \Bigl(\alpha_{n,1} n^{-1} e_{n,2}+\sum_{k=2}^\infty \alpha_{n,k}e_{n,k+1}\Bigr).
$$
Then $A$ is  positively homogeneous, additive, Lipschitz mapping that preserves finite suprema, such that 
$d(A)=0$ and $r_x(A)=1$ for all nonzero $x\in C$. Moreover, $\si_{ap}(A)=\{0,1\}$ and so $\si_{ap}(A)$ does not contain the whole interval \\ $\bigl[d(A), \inf\{r_x(A): x\in C, \; x\ne 0\}\bigr]$.
}
\label{interval_not}
\end{example}

By Theorem \ref{appl_to_matrices},
$r(A)=\max\{t:t\in\si_{ap}(A)\}$ if  $X$ is a normed vector lattice, $C\subset X_+$  a non-zero max-cone and $T:C\to C$  a mapping, which is Lipschitz, positively homogeneous and preserves finite suprema.
 We show below in Theorem \ref{dAisMin} that under these assumptions we also have that $d(A)=\min\{t: t\in\si_{ap}(A)\}$. In the proof we will need the following lemmas (\cite[Lemma 3.1, Lemma 3.2]{MP17}). The first one is  based on the inequality (\ref{Birk_inq}). 
\begin{lemma}
\label{Birk}
Let $X$ be a normed vector lattice and let $x_1,\dots,x_n,y_1,\dots,y_n\in X$. Then
$$
\Bigl\|\bigvee_{j=1}^n x_j-\bigvee_{j=1}^n y_j\Bigr\|\le
\sum_{j=1}^n \|x_j-y_j\|.
$$
\end{lemma}
\begin{lemma} Let $X$ be a vector lattice and $x_j, y_j \in X$ 
for $j= 1, \ldots, n$. Then 
\be
\bigvee_{j=1}^n x_j - \bigvee_{j=1}^n y_j  \le \bigvee_{j=1}^n (x_j -y_j) .
\label{supineq}
\ee
If, in addition, $X$ is a normed vector lattice and  $x_j \ge y_j \ge 0$ 
for $j= 1, \ldots, n$, then 
\be
\Bigl\|\bigvee_{j=1}^n x_j - \bigvee_{j=1}^n y_j \Bigr\| \le \Bigl\|\bigvee_{j=1}^n (x_j -y_j) \Bigr\|.
\label{supineq2}
\ee
\label{sup}
\end{lemma}

The following result is one of the main results of this section.

\begin{theorem} Let $X$ be a normed vector lattice, let $C\subset X_+$ be a non-zero max-cone. Let $A:C\to C$ be a mapping which is bounded, positively homogeneous and preserves finite suprema. Then
$d(A)\in\si_{ap}(A)$. 

If, in addition, $A$ is  Lipschitz, then $d(A)=\min\bigl\{t: t\in\si_{ap}(A)\bigr\}$. 
\label{dAisMin}
\end{theorem}
\begin{proof}
If $d(A)=0$ then $m(A)=0$ and $0\in\si_{ap}(A)$.

Let $d(A)> 0$. Without loss of generality we may assume that $d(A)=1$.

Let $\e>0$. We show that there exists $ w \in C$, $w\ne 0$ such that $\frac{\|Aw-w\|}{\|w\|}\le\e$.

Let $n\in\NN$ satisfy $n>\max\{2,16\e^{-2}\}$ and $m(A^n)>0$. Find $\de>0$ such that $(1-\de)^n\ge \frac{1}{2}$ and $(1+\de)^n\le 2$.
Find $N_1\in\NN$, $n|N_1$, $N_1 \ge 2n$ and 
$$
(1-\de)^N<m(A^N)<(1+\de)^N
$$
for all $N\ge N_1$.
Find $s\in\NN$, $s\ge 16$ with
$m(A^n)>2^{-s/4}$ and let $N=sN_1$.

Find $x\in C$, $\|x\|=1$ such that $\|A^Nx\|<(1+\de)^N\le 2^{N/n}$. We also have  $\|A^Nx\|\ge m(A^N)>(1-\de)^N\ge 2^{-N/n}$.

Consider the vectors $x, A^nx, A^{2n}x,\dots,A^Nx$.
Write for short $a_j=\|A^{jn}x\|$ for $j=0,\dots,N/n$.
\bigskip

\noindent{\it Claim.}
There exists $m, 1\le m<N/n$ with
$$
a_m\ge \frac{1}{4}\max\{a_{m-1},a_{m+1}\}.
$$
\medskip

Suppose on the contrary that $4a_m<\max\{a_{m-1},a_{m+1}\}$ for all $m=1,\dots,N/n-1$.

If $a_{j+1}\ge a_j$ for some $j\le \frac{N}{n}-2$, then this condition for $j+1$ means that $a_{j+2}>4a_{j+1}$. By induction we get
$$
a_j\le a_{j+1}\le a_{j+2}\le \cdots\le a_{N/n}.
$$
Similarly, if $a_{j+1}\le a_j$ for some $j\ge 1$ then $a_{j-1}>4a_j$. By induction, we get
$$
a_{j+1}\le a_{j}\le a_{j-1}\le\cdots\le a_0=1.
$$
So there exists $k, 0\le k\le N/n$ such that
$$
a_0\ge a_1\ge\cdots\ge a_{k-1}\ge a_k\le a_{k+1}\le\cdots\le a_{N/n}.
$$
Moreover,
$$
1=a_0\ge 4a_1\ge 4^2a_2\ge\cdots\ge 4^{k-1}a_{k-1}
$$
and 
$$
a_{N/n}\ge 4a_{N/n-1}\ge\cdots\ge 4^{\frac{N}{n}-k-1}a_{k+1}.
$$

If $k\ge \frac{N_1}{n}+1$ then
$$
a_{k-1}=\|A^{(k-1)n}x\|\ge m(A^{(k-1)n})> (1-\de)^{(k-1)n}\ge\frac{1}{2^{k-1}},
$$
a contradiction with the estimate $a_{k-1}<\frac{1}{4^{k-1}}$.

If $k<\frac{N_1}{n}+1$ then $k+1\le\frac{2N_1}{n}=\frac{2N}{sn}$. We have
$$
a_{N/n} >
a_{k+1}\cdot 4^{\frac{N}{n}-k-1}\ge
\bigl(m(A^n)\bigr)^{k+1}4^{\frac{N}{n}-k-1}\ge
2^{-\frac{s}{4}\frac{2N}{sn}}\cdot 4^{\frac{N}{n}-\frac{2N}{sn}}
$$
$$
\ge
2^{-\frac{N}{2n}}\cdot 2^{\frac{2N}{n}}\cdot 2^{\frac{-4N}{sn}}\ge
2^{\frac{N}{n}(\frac{3}{2}-\frac{1}{4})}=
2^\frac{5N}{4n}.
$$
However, $a_{N/n}=\|A^Nx\|<(1+\de)^N\le 2^{N/n}$, a contradiction.
\bigskip

\noindent{\it Continuation of the Proof of Theorem \ref{dAisMin}.}
Let $m, 1\le m<N/n$ satisfy $\|A^{mn}x\|\ge\frac{1}{4}\max\bigl\{\|A^{(m-1)n}x\|,\|A^{(m+1)n}x\|\bigr\}$.

Let $y=\bigvee_{j=(m-1)n}^{(m+1)n-1}A^jx$. Then
$Ay=\bigvee_{j=(m-1)n+1}^{(m+1)n}A^jx$ and by Lemma \ref{Birk} we have $ \|Ay-y\|\le\|A^{(m-1)n}x\|+\|A^{(m+1)n}x\|\le 8\|A^{mn}x\|$. 
If $\|y\|\ge 8\e^{-1}\|A^{mn}x\|$ then $\frac{\|Ay-y\|}{\|y\|}\le \e$ and we are done. So assume that $\|y\|< 8\e^{-1}\|A^{mn}x\|$.

Let
$$
u=\Bigl(\frac{1}{n}A^{(m-1)n}x\vee \frac{2}{n}A^{(m-1)n+1}x\vee \cdots 
\vee
\frac{n-1}{n}A^{mn-2}x\vee A^{mn-1}x\Bigr)
$$
$$
\vee
\Bigl(\frac{n-1}{n}A^{mn}x\vee \cdots\vee \frac{2}{n}A^{(m+1)n-3}x\vee
\frac{1}{n}A^{(m+1)n-2}x\Bigr).
$$
We have $\|u\|\ge\frac{n-1}{n}\|A^{mn}x\|\ge\frac{1}{2}\|A^{mn}x\|$. Furthermore, by Lemma \ref{sup}
$$
Au-u\le \frac{1}{n}\Bigl(A^{(m-1)n}x\vee A^{(m-1)n+1}x\vee\cdots\vee A^{(m+1)n-1}x\Bigr) = \frac{y}{n},
$$
and similarly, $u-Au\le \frac{y}{n}$. Hence
$\|Au-u\|\le n^{-1}\|y\|< 8n^{-1}\e^{-1}\|A^{mn}x\|$ and
$$
\frac{\|Au-u\|}{\|u\|}\le\frac{16}{n\e}< \e.
$$

Since $\e>0$ was arbitrary, $1\in\si_{ap}(A)$.

If, in addition, $A$ is  Lipschitz, then $d(A)=\min\{t: t\in\si_{ap}(A)\}$ by Proposition \ref{basic}.
\end{proof}

 By replacing  $\vee$ with $+$ in the proof above and by suitably adjusting some estimates 
the following theorem also follows. We include some details of the proof for the sake of clarity.

\begin{theorem}
Let $X$  be a normed space, $C\subset X$ a non-zero normal wedge and let $A:C\to C$ be positively homogeneous, additive and   bounded. 
Then $d(A)\in\sigma_{ap}(A)$. 

If, in addition, $A$ is Lipschitz, then  $d(A)=\min\bigl\{t: t\in\si_{ap}(A)\bigr\}$.
\label{dAnormal}
\end{theorem}
\begin{proof}
As in the proof of Theorem \ref{dAisMin} we may assume that $d(A)=1$. Let $\e>0$ and let $n\in \NN$ satisfy 
$n > \max\{2, 32M\e ^{-2} \}$ and $m(A^n)>0$, where $M$ is the normality constant of $C$. If $N \in \NN$ is chosen as   in the proof of Theorem \ref{dAisMin}, then it follows in the same way that there exists $m\in \NN, 1\le m<N/n$ such that
 $$\|A^{mn}x\|\ge\frac{1}{4}\max\bigl\{\|A^{(m-1)n}x\|,\|A^{(m+1)n}x\|\bigr\}.$$
Let $y=\sum _{j=(m-1)n} ^{(m+1)n-1} A^j x$. Then $Ay-y = A^{(m+1)n}x -A^{(m-1)n}x$ and so
$\|Ay-y\| \le 8 \|A^{mn}x\|$. Without loss of generality we may assume that  $\|y\|< 8\e^{-1}\|A^{mn}x\|$.
Let
$$
u=\Bigl(\frac{1}{n}A^{(m-1)n}x+ \frac{2}{n}A^{(m-1)n+1}x+\cdots 
+
\frac{n-1}{n}A^{mn-2}x+ A^{mn-1}x\Bigr)
$$
$$
+\Bigl(\frac{n-1}{n}A^{mn}x+ \cdots + \frac{2}{n}A^{(m+1)n-3}x +
\frac{1}{n}A^{(m+1)n-2}x\Bigr).
$$ 
Then $2\|u\| \ge \|A^{mn}x\| $ and 
$$\|Au-u\| = \frac{1}{n} \Bigl\|A^{(m+1)n-1}x+ \cdots + A^{mn}x - \left(A^{mn-1}x + \cdots + A^{(m-1)n} \right)\ \Bigr\|$$
$$\le  \frac{1}{n} \left (\bigl\|A^{(m+1)n-1}x+ \cdots + A^{mn}x \bigr\|+ \bigl\| A^{mn-1}x + \cdots + A^{(m-1)n}  \bigr\| \right) \le \frac{2M\|y\|}{n} $$
$$< \frac{16M}{n\e}\|A^{mn}x\| \le  \frac{32M}{n\e}\|u\| < \e \|u\|.$$
Since $\e>0$ was arbitrary, $1\in\si_{ap}(A)$.
\end{proof}
There are interesting non-trivial examples of operators to which Theorem \ref{dAnormal} applies. In particular, it applies to the $C$-linear Perron-Frobenius operators from \cite[Section 5]{MN10} and \cite[Sections 5 and 6]{N01}.

\begin{remark} {\rm In  \cite[Example 2]{MP17} there is an example of a Banach lattice $X$ and a closed max-cone 
$C \subset X_+$,  which is normal and convex (a normal wedge), and a  positively homogeneous, additive mapping $A:C \to C$ that preserves all suprema and satisfies $\|A\| = 1$.
However, $A$ is not Lipschitz, $d(A)=r(A)=0$ and $\{0, 1\} \subset \si_{ap}(A)$. So,   as pointed out in  \cite{MP17}, the Lipschitzity of $A$ is necessary for the property $r(A)=\max\{t: t\in\si_{ap}(A)\}$ to hold. 
}
\end{remark}

The following example shows that the Lipschitzity of $A$ is necessary for the property $d(A)=\min\{t: t\in\si_{ap}(A)\}$ to hold in Theorems \ref{dAisMin} and \ref{dAnormal}.

\begin{example}
{\rm
Let $X$ be as in Example \ref{interval_not}.
Let
$$
C=\Bigl\{\sum_{n,k\in\NN}\alpha_{n,k}e_{n,k}\in X_+:  \alpha_{n,k} \ge 0,\; \alpha_{n,2}\le n^{-1}\alpha_{n,1}\hbox{ for all } n\in\NN, k \in \NN \Bigr\}.
$$ 
Then $C$ is a max-cone (moreover $C$ is a convex normal cone). Let $A:C\to C$ be defined by
$$
A\Bigl(\sum_{n,k\in\NN}\alpha_{n,k}e_{n,k}\Bigr)=
\sum_{n\in\NN} \Bigl(\alpha_{n,1}e_{n,1}+n^{-1}\alpha_{n,1}e_{n,2}+2n\alpha_{n,2}e_{n,3}+\sum_{k=3}^\infty2\alpha_{n,k}e_{n,k+1}\Bigr).
$$
Clearly $A$ is bounded, $\|A\|=2$ and $A$ is a positively homogeneous, additive mapping on $C$ that preserves finite suprema. 
We have $\|Ae_{n,1}-e_{n,1}\|=\|n^{-1}e_{n,2}\|=n^{-1}$ and $\|A^2 e_{n,1} - Ae_{n,1}\| =\|2e_{n,3}\|=2$  for all $n$, so $1\in\sigma_{ap}(A)$ and $A$ is not Lipschitz. On the other hand, it is easy to see that $d(A)=2$.

}
\end{example}

\section{Spectral mapping theorems for point and approximate point spectrum.}
In this section we generalize the spectral mapping theorem in max-algebra (see \cite[Theorem 3.4]{MP15} and also \cite[Theorem 3.6]{KSS12} ) to the infinite dimensional setting. The main results of this section are Corollary \ref{without_constant}, Theorem \ref{as_good_as_possible} and Theorem  \ref{main_spectral_mapping}. 

Let $X$ be a Riesz space (i.e., a vector lattice). Let $C\subset X_+$ be a nonzero max-cone and $A:C\to C$ a positively homogeneous mapping that preserves finite suprema.
Let $\PP_+$ denote the set of all polynomials $q(z)=\sum_{j=0}^n\al_jz^j\in\PP_+$  with $\al_j\ge 0$ for all $j$. For $q\in \PP_+$ and $t\ge 0$ let $q_\vee(t)=\max_j \al_j t^j$ (i.e., $q_\vee $ is the maxpolynomial corresponding to $q$).
Define $q_\vee(A):C\to C$ by
$$
q_\vee(A)x=\bigvee_{j=1}^n \al_jA^jx\qquad(x\in C).
$$

\begin{lemma} Let $X$ be a vector lattice and  let $C\subset X_+$ be a nonzero max-cone. Let $A:C\to C$ be a positively homogeneous mapping that preserves finite suprema
and  $q=\sum_{j=0}^n\al_jz^j\in\PP_+$. Then
$$
q_\vee(\si_p(A))\subset\si_p(q_\vee(A)).
$$
\label{obvious_part}
\end{lemma}

\begin{proof}
Let $t\ge 0$, $x\in C$ and $Ax=tx$. Then $A^jx=t^jx$ for all $j\ge 0$,
 and so $q_\vee(A)x=q_\vee(t)x$. Hence $q_\vee(\si_p(A))\subset\si_p(q_\vee(A)).$
\end{proof}

\begin{lemma}  Let $X$ be a vector lattice and  let $C\subset X_+$ be a nonzero max-cone. Let $A:C\to C$ be a positively homogeneous, finite suprema preserving mapping.
Assume that $q(z)=\sum_{j=1}^n\al_jz^j\in\PP_+$, $q_\vee(1)=1$ and $1\in\si_p(q_\vee(A))$. Then $1\in\si_p(A)$.
\label{almost_done}
\end{lemma} 

\begin{proof}
Since $q_\vee(1)=1$, we have $\al_j\le 1$ for all $j$ and there exists $m, 1\le m\le n$ with $\al_m=1$.

Let $x\in C$ be a nonzero vector satisfying $q_\vee(A)x=x$.
Set $y=\bigvee_{j=0}^{m-1}A^jx$.
We have $A^mx\le q_\vee(A)x =  x$. Consequently, $A^{m+j}x\le A^jx$ for all $j\in\NN$.
Thus $y=\bigvee_{j=0}^{n-1} A^j x$. We have $Ay =\bigvee_{j=1}^{n} A^j x \ge q_\vee(A)x =x$.
Thus
$$Ay=x\vee \bigvee_{j=1}^m A^jx \ge \bigvee_{j=0}^{m-1} A^jx =y.$$
Conversely,  $A^mx \le x \le y$ and
$$
Ay=
\bigvee_{j=1}^{m-1} A^jx\vee A^mx \le 
y.
$$
Hence $Ay=y$ and $1\in \si_p(A)$.
\end{proof}

\begin{corollary} Let $X$ be a vector lattice and  let $C\subset X_+$ be a nonzero max-cone. Let $A:C\to C$ be a positively homogeneous, finite suprema preserving mapping
and  $q=\sum_{j=1}^n\al_jz^j\in\PP_+$ a non-zero polynomial. Then
$$
\si_p(q_\vee(A))=q_\vee(\si_p(A)).
$$
\label{without_constant}
\end{corollary}

\begin{proof}
The inclusion $\supset$ was proved above.
\smallskip

$\subset$: Let $t\ge 0$, $s=q_\vee(t)$ and $s\in\si_p(q_\vee(A))$. 
If $t=0=s$, $x\in C$, $x\ne 0$ and $q_\vee(A)x=0$, then there exists $m, 1\le m\le n$ with
$\al_m\ne 0$. So $\al_m A^m x=0$ and $A^mx=0$.
Find $k, 1\le k\le m-1$ with $A^kx\ne 0$ and $A^{k+1}x=0$. Then $A^kx$ is a nonzero eigenvector and $0\in\si_p(A)$.

Let $s=q_\vee(t)\ne 0$. Then $t\ne 0$. Consider the mapping $A'=t^{-1}A$ and polynomial
$p(z)=s^{-1}q(tz)=\sum_{j=1}^n\frac{\al_jt^j}{s}z^j$. Then $p_\vee(1)=1$.
Let $x\in C$ be a nonzero vector satisfying $q_\vee(A)x=sx$. Then 
$$
p_\vee(A')x=
\bigvee_{j=1}^n \frac{\al_j t^j}{s}\Bigl(\frac{A}{t}\Bigr)^jx=x.
$$
By  Lemma \ref{almost_done} there exists $y\in C$, $y\ne 0$ and $A'y=y$. Hence $Ay=ty$ and $t\in\si_p(A)$.
\end{proof}

\begin{lemma}
Let $X$ be a vector lattice, $x,y\in X_+$, $s>1$ and $x\vee y=sx$. Then $y=sx$.
\label{good}
\end{lemma}

\begin{proof}
Let $k$ be the smallest integer satisfying $k(s-1)\ge 1$. We prove by induction that $y\ge j(s-1)x$ for $j=1,\dots,k$.

We have $x+y\ge x\vee y=sx$, and so $y\ge (s-1)x$.
Let $1\le j\le k-1$ and suppose that $y\ge j(s-1)x$.
Since $x\vee y=sx$, we have
$$
(x-j(s-1)x) \vee (y-j(s-1)x) = sx -j(s-1)x,
$$
and so
$$
x\vee \frac{y-j(s-1)x}{1-j(s-1)} =x\frac{s-j(s-1)}{1-j(s-1)}.
$$
By the induction assumption for $j=1$ applied to $x\vee y' =s'x$, where $y'=\frac{y-j(s-1)x}{1-j(s-1)}$ and $s'=\frac{s-j(s-1)}{1-j(s-1)}$ we have
$$
\frac{y-j(s-1)x}{1-j(s-1)}\ge x\Bigl(\frac{s-j(s-1)}{1-j(s-1)}-1\Bigr).
$$
By multiplying both sides with $1-j(s-1)$  we obtain
$$
y-j(s-1)x\ge (s-1)x
$$
and so
$$
y\ge (j+1)(s-1)x.
$$
By induction, $y\ge j(s-1)x$ for all $j=1,\dots,k$. Hence $y\ge k(s-1)x\ge x$ and
$y=x\vee y=sx$.
\end{proof}


\begin{proposition}
 Let $X$ be a vector lattice and  let $C\subset X_+$ be a nonzero max-cone. Let $A:C\to C$ be a positively homogeneous, finite suprema preserving mapping.
If $\al>0$ and $q(z)=\al+z$, then
$$
\si_p(q_\vee(A))\cap (\al,\infty) = q_\vee(\si_p(A))\cap(\al,\infty).
$$
\label{max_linear}
\end{proposition}

\begin{proof}
The inclusion $\supset$ follows from Lemma \ref{obvious_part}.

$\subset$: Let $t>\al$ and $t\in\si_p(q_\vee(A))$. So there exists a nonzero $x\in C$ with
$q_\vee(A)x= \al x\vee Ax=tx$. So  $x\vee \al^{-1}Ax=\al^{-1}tx$, where $\al^{-1}t>1$. By Lemma \ref{good}
for $y=\al^{-1}Ax$ we have
$\al^{-1}Ax=\al^{-1}tx$, and so $Ax=tx$. Hence $t\in\si_p(A)$.
\end{proof}
Now the following result follows. 
\begin{theorem}  Let $X$ be a vector lattice and  let $C\subset X_+$ be a nonzero max-cone. Let $A:C\to C$ be a positively homogeneous, finite suprema preserving mapping and
let $q(z)=\sum_{j=0}^n\al_jz^j\in\PP_+$. Then
$$
q_\vee(\si_p(A))\subset
\si_p(q_\vee(A))\subset
q_\vee(\si_p(A))\cup\{\al_0\}.
$$
\label{as_good_as_possible}
\end{theorem}

\begin{proof}
The first inclusion follows from Lemma \ref{obvious_part}.

If $\alpha_0=0$ then the second inclusion follows from Corollary \ref{without_constant}. Let $\alpha_0\ne 0$ and $t>\alpha_0$, $t\in\sigma_p(q_\vee(A))$. So there exists a nonzero $x\in C$ such that $tx=q_\vee(A)x=\alpha_0x\vee y$, where $y=\bigvee_{j=1}^n\alpha_jA^jx$. By Lemma \ref{good}, we have $y=tx$ and $t\in\sigma_p(\bigvee_{j=1}^n\alpha_jA^j)$. By Corollary \ref{without_constant}, there exists $s\in\sigma_p(A)$ with $t=\max\{\alpha_js^j:1\le j\le n\}=
\max\{\alpha_js^j:0\le j\le n\}=q_\vee(s)$. So $t\in q_\vee(\sigma_p(A))$.
\end{proof}

\begin{remark}{\rm Under the assumptions of Theorem \ref{as_good_as_possible}
it is possible for $q(z)=\sum_{j=0}^n\al_jz^j$ that $\si_p(q_\vee(A))\ne q_\vee( \si_p(A))$.
Consider the Banach lattice $\ell^\infty$ with natural order and let $C$ be the positive cone. Let $(e_n)$ be the standard basis in $\ell^\infty$ and define a mapping $A:C\to C$ by $A(\sum_n\gamma_ne_n)=\sum_nn^{-1}\gamma_ne_{n+1}$. 
Then $\sigma_p(A)=\emptyset$. Let $q_\vee(z)=1\vee z$. Then for $y=\sum_{n}e_n$ we have $q_\vee(A)y=y\vee Ay=y$. Hence  
$\si_p(q_\vee(A))=\{1\}\ne q_\vee(\si_p(A))=\emptyset$.
}
\end{remark}

\bigskip

In the following, $X$ will be a normed vector lattice and $C\subset X_+$ a non-zero max cone. Let $A:C\to C$ be positive homogeneous, Lipschitz and finite suprema preserving.
The spectral mapping theorem for the approximate point spectrum (see Theorem  \ref{main_spectral_mapping} below) 
 can be proved similarly as the above results by repeating similar arguments. However, one can also apply the following standard construction.

Denote by $\ell^\infty(X)$ the set of all bounded sequences $(x_j)_{j=1}^\infty$ of elements of $X$. With the norm $\|(x_j)\|=\sup_j \|x_j\|$ and order $(x_j)\le (y_j)\Leftrightarrow x_j\le y_j\hbox{ for all }j$, $\ell^\infty(X)$ is again a normed  vector lattice.
Let $C^\infty\subset \ell^\infty(X)$ be the set of all bounded sequences of elements of $C$ and let $A^\infty:C^\infty\to C^\infty$ be defined by $A^\infty((c_j))=(Ac_j)$. 
Let $c_0(X)$ be the set of all null sequences $(x_j)$ of elements of $X$, $\|x_j\|\to 0$.
Clearly $c_0(X)$ is an ideal in $\ell^\infty(X)$. Let $\widetilde X=\ell^\infty(X)/c_0(X)$. Then $\widetilde X$ is again a normed lattice. Let $\widetilde C=(C^\infty+c_0(X))/c_0(X)$
and $\widetilde A:\widetilde C\to\widetilde C$ be the natural quotient mapping, which is well defined since $A$ is Lipschitz. 
Then $\widetilde C$ is a max cone and $\widetilde A$ is positive homogeneous and finite suprema preserving. Moreover, it is easy to show that
$$
\si_{ap}(A)=\si_p(\widetilde A).
$$
Thus the following result follows.

\begin{theorem}  Let $X$ be a normed vector lattice and  let $C\subset X_+$ be a nonzero max-cone. Let $A:C\to C$ be a Lipschitz, positively homogeneous, finite suprema preserving mapping.
Let $q(z)=\sum_{j=0}^n\al_jz^j\in\PP_+$. Then
$$
q_\vee(\si_{ap}(A))\subset
\si_{ap}(q_\vee(A))\subset
q_\vee(\si_{ap}(A))\cup\{\al_0\}.
$$
If $\al_0=0$ then $q_\vee(\si_{ap}(A))=
\si_{ap}(q_\vee(A)).$
\label{almost_there}
\end{theorem}

\begin{corollary}  Let $X$ be a normed vector lattice and  let $C\subset X_+$ be a nonzero max-cone. Let $A:C\to C$ be a  Lipschitz, positively homogeneous, finite suprema preserving mapping.
If $q\in\PP_+$, $q=\sum_{j=0}^{\deg q}\al_jz^j$, then
\be
 r (q_\vee (A))  =   q_\vee (r (A)). 
\label{eq_r_q}
\ee 
\label{sp_equality}
\end{corollary}
\begin{proof} 

By Theorems \ref{almost_there} and \ref{appl_to_matrices} 
it follows that
$$q_\vee (r (A)) \le   r (q_\vee (A))  \le  \max\{ q_\vee (r(A)) ,\alpha_0\} = q_\vee (r (A)),$$
which completes the proof.
\end{proof}

The equality (\ref{eq_r_q})  suggests that
the situation for the approximate point spectrum is even better. Indeed, the equality $\si_{ap}(q_\vee(A))=q_\vee(\si_{ap}(A))$ is true for all polynomials from $\PP _+$ as we prove below in Theorem \ref{main_spectral_mapping}.

\begin{proposition} Let $X$ be a normed vector lattice and  let $C\subset X_+$ be a nonzero max-cone. Let $A:C\to C$ be a  Lipschitz, positively homogeneous, finite suprema preserving mapping.
Let $q(z)=\al+z\in\PP_+$. Then $\si_{ap}(q_\vee(A))=q_\vee(\si_{ap}(A))$.
\label{first}
\end{proposition}
\begin{proof}
If $\al=0$ then the statement is clear.

Let $\al\ne 0$. Without loss of generality we may assume that $\al=1$. Let $B=I\vee A$.

We know that
$\si_{ap}(B)\supset q_\vee(\si_{ap}(A))$ and
$$
\si_{ap}(B)\cap (1,\infty)=q_\vee(\si_{ap}(A))\cap(1,\infty).
$$
Moreover, $m(B)\ge 1$ and $\si_{ap}(B)\subset [1,\infty)$.

Let $1\in\si_{ap}(B)$. For each $j\in\NN$ we have $1\in\si_{ap}(B^j)$. Let $(x_k)$ be a sequence of unit vectors in $C$ satisfying
$\|B^jx_k-x_k\|\to 0$. Thus $m(B^j)\le 1$ and $m(A^j)\le m(B^j)\le 1$. Hence $d(A)=\lim m(A^j)^{1/j}\le 1$
and so  $q_\vee(d(A))=1$. Since $d(A)\in\si_{ap}(A)$ by Theorem \ref{dAisMin} it follows that $1\in q_\vee(\si_{ap}(A))$ and so
$\si_{ap}(q_\vee(A))=q_\vee(\si_{ap}(A))$.
\end{proof}
Now the following result follows from Theorem \ref{almost_there}, Proposition \ref{first} and Theorem  \ref{dAisMin}.
\begin{theorem}  Let $X$ be a normed vector lattice and  let $C\subset X_+$ be a nonzero max-cone. Let $A:C\to C$ be a  Lipschitz, positively homogeneous, finite suprema preserving mapping.
Let $q(z)=\sum_{j=0}^n \al_jz^j\in\PP_+$. Then
$$
\si_{ap}(q_\vee(A))=q_\vee(\si_{ap}(A))
$$
and so 
$$
d(q_\vee(A))=q_\vee(d(A)).
$$
\label{main_spectral_mapping}
\end{theorem}

In particular, the results above apply to the following two classes of examples from \cite{MN02} and \cite{MP17}.

\begin{example} 
\label{max_cont}
{\rm  Given $a>0$, consider the following
max-type kernel operators $A: C[0,a] \to C[0,a]$ of the form
$$(A(x))(s)=\max_ {t\in [\alpha (s), \beta (s)]}{k(s,t)x(t)},$$
where $x\in C[0,a]$ and $\alpha, \beta:[0,a]\to[0,a]$ are given continuous functions satisfying $\alpha \le \beta$.
The kernel  $k: S \to [0, \infty)$ is a given non-negative continuous function, where $S$
 denotes the compact set
$$S=\{(s,t)\in [0,a]\times[0,a]: t \in [\alpha (s), \beta (s)]\}.$$
It is clear that for $C=C_+ [0,a]$ it holds $AC \subset C$. We will denote the restriction $A|_C$ again by $A$.
The eigenproblem of these operators arises in the study of periodic solutions of a class of 
differential-delay equations
$$\varepsilon y^{\prime}(t)=g(y(t),y(t-\tau)), \quad \tau=\tau(y(t)),$$
with state-dependent delay (see e.g. \cite{MN02}). 

By  \cite[Proposition 4.8]{MN02} 
and its proof, the operator $A:C \to C$ is a positively homogeneous, Lipschitz mapping that preserves finite suprema. Hence 
$r(A)= \max \{t : t \in \sigma_{ap}(A)\}$ and $d(A)= \min \{t : t \in \sigma_{ap}(A)\}$. By  \cite[Theorem 4.3]{MN02} 
it also holds that 
$r(A)= \lim _{n \to \infty} b_n ^{1/n} =\inf _{n\ge 1} b_n ^{1/n}$, where
$b_n =\|A^n\|= \max _{\sigma \in S_n} k_n (\sigma)$,
$$k_n  (\sigma)=k(s_0, s_1)k(s_1, s_2) \cdots k (s_{n-1}, s_n) $$
and
$$S_n =\bigl\{(s_0, s_1,s_2,  \ldots , s_n): s_0 \in [0,a], s_i \in [\alpha(s_{i-1}), \beta (s_{i-1})], i=1,2, \ldots , n \bigr\}.$$ 
On the other hand,
$$d(A)=\lim_{n\to\infty}d_n ^{1/n}= \sup _{n\in N} d_n ^{1/n},$$ 
where
$$d_n = m(A^n) = \inf\bigl\{ \max _{\sigma \in S_n} k_n (\sigma) x(s_n) :x\in C, \|x\|=1\bigr\}.$$

}
\end{example}

We also point out the following related example from \cite{MP17}.

\begin{example} 
\label{sup_bounded}
{\rm
Let $M$ be a nonempty set and let $X$ be the set of all bounded real functions on $M$.
With the norm $\|f\|_{\infty}=\sup\{|f(t)|:t\in M\}$ and natural operations, $X$ is a normed vector lattice.
Let $C=X_+$ 
and
let $k:M\times M\to [0,\infty)$ satisfy $\sup \bigl\{k(t,s):t,s\in M\bigr\}<\infty$.
Let $A:C\to C$ be defined by $(Af)(s)=\sup\{k(s,t)f(t):t\in M\}$ and so $\|A\|=\sup \bigl\{k(t,s):t,s\in M\bigr\}$. Clearly $C$ is a max-cone, $A$ is bounded, positive homogeneous and preserves finite  suprema. 
Moreover, $A$ is Lipschitz. So 
we have that $r(A)= \max \{t : t \in \sigma_{ap}(A)\}$ and  $d(A)= \min \{t : t \in \sigma_{ap}(A)\}$.

In particular, if $M$ is the set of all natural numbers $\NN$, our results apply to infinite bounded non-negative matrices $k=[k(i,j)]$ (i.e., $k(i,j) \ge 0$ for all $i,j \in \NN$ and    $\|k\|_{\infty}=\sup _{i,j \in \NN} k(i,j) < \infty$). In this case, $X= l^{\infty}$ and 
$C=l^{\infty}_+ $ and $\|A\| =\|k\|_{\infty}$.
}
\end{example}

\section{Application to inequalities involving Hadamard products.}

Throughout this section let $X$, $C$ and all the mappings  $A, B, A_1, \ldots, A_m, A_{11}, \ldots , A_{mn}$  that map $C$ to $C$ be as in Example \ref{max_cont} (where the functions $\alpha$ and $\beta$ are fixed - the same for all operators  $A, B, A_1, \ldots, A_m, A_{11}, \ldots , A_{mn}$) or let $X$, $C$ and  all the mappings  $A, B, A_1, \ldots, A_m, A_{11}, \ldots , A_{mn}$  that map $C$ to $C$ be as in Example \ref{sup_bounded}. We  denote the set of such mappings by $\mathcal{C}$.  

In this section we apply  (\ref{eq_r_q})
 to prove some new inequalities on Hadamard products (Theorem \ref{powerineq}) by applying an idea from \cite{DP16} .  
Let $A\circ B$ denote the Hadamard (or Schur) product of mappings $A$ and $B$ from $\mathcal{C}$, i.e., 
$A\circ B \in \mathcal{C}$ is a mapping with a kernel $k(s,t)h(s,t)$, where $k$ and $h$ are the kernels of $A$ and $B$, respectively.
Similarly, for $\gamma >0$ let $A^{(\gamma)}$ denote the Hadamard (or Schur) power of $A$, i.e., a mapping with a kernel $k^{\gamma}(s,t)$. 

The following result was stated in \cite[Theorem 4.1]{P12} 
in the special case of $n\times n$ non-negative matrices and was essentially proved in \cite{P06}. 
It follows from \cite[Theorem 5.1 and Remark 5.2]{P06} 
and the fact that for $A_1, \ldots, A_m, A \in \mathcal{C}$ and $\gamma>0$ we have
$$A_1 ^{(\gamma)} \cdots   A_m ^{(\gamma)}=(A_1  \cdots   A_m) ^{(\gamma)} \;\;\mathrm{and}\;\; \|A^{(\gamma)}\| = \|A\|^{\gamma} $$
and consequently $r (A^{(\gamma)}) = r (A)^{\gamma}$.  
Observe also that $A\le B$ implies $r(A)\le r(B)$.

\begin{theorem} Let $A_{i j} \in \mathcal{C}$ for $i=1, \ldots , n$ and  $j=1, \ldots , m$  and let
$\alpha _1$, $\alpha _2$,..., $\alpha _m$ be positive numbers. 
Then we have

$$\left(A_{1 1}^{(\alpha _1)} \circ A_{1 2}^{(\alpha _2)} \circ \cdots \circ A_{1 m}^{(\alpha _m)}\right)  \ldots 
\left(A_{n 1}^{(\alpha _1)} \circ A_{n 2}^{(\alpha _2)} \circ \cdots \circ A_{n m}^{(\alpha _m)}\right)$$ 
$$\le (A_{1 1}  \cdots   A_{n 1})^{(\alpha _1)} \circ (A_{1 2}  \cdots   A_{n 2})^{(\alpha _2)} \circ \cdots 
\circ (A_{1 m}  \cdots  A_{n m})^{(\alpha _m)} $$
and

$$r \left(\left(A_{1 1}^{(\alpha _1)} \circ \cdots \circ A_{1 m}^{(\alpha _m)}\right) \ldots 
\left(A_{n 1}^{(\alpha _1)} \circ \cdots \circ A_{n m}^{(\alpha _m)} \right) \right)$$
\be
\label{spectral2}
\le r  \left( A_{1 1}\cdots  A_{n 1} \right)^{\alpha _1} \cdots 
r \left( A_{1 m} \cdots   A_{n m}\right)^{\alpha _m} .
\ee
\label{DPmax}
\end{theorem}

\begin{remark} {\rm An analogue to (\ref{spectral2}) for $\|\cdot\|$ also holds. 
As pointed out in 
\cite{P06}, \cite{P12} and \cite{MP12}, 
Theorem  \ref{DPmax}
 is in fact a result on the generalized  and joint spectral radius in max-algebra (see also \cite{P16a}). The logarithm of the latter is also known as the maximal Lyapunov exponent in max algebra and is important in the study of certain discrete event systems (see e.g. the references cited in \cite{MP12}). 
}
\end{remark}

 Inequalities (\ref{maxmixmax}) and  (\ref{kvmax}) 
below were established in  \cite[Corollary 4.8]{P12}
 in the special case of $n \times n$ matrices,
while the inequality (\ref{products}) 
is a max-algebra version of \cite[Corollary 3.3]{P16b} and \cite[Theorem 3.2]{DP16}.
 The proofs of these inequalities are similar to the proofs of the results from  \cite{P12}, \cite{DP16} and \cite{P16b} 
and are included for the convenience of readers. In the proof we use the fact that 
\be
 r (A B) =  r (B A).
\label{commute} 
\ee

\begin{corollary}
\label{max_ineq}

Let  $A_1, \ldots, A_m, A, B \in \mathcal{C}$ and let 
$P_i =A_i  A_{i+1}  \cdots  A_m  A_1  \cdots  A_{i-1}$ for $i=1, \ldots, m$.  
Then the following inequalities hold:
\be
r (A _1  \circ \cdots \circ A _m) \le r (P_1  \circ \cdots \circ P_m)^{1/m} \le r (A_1   \cdots  A_m),
\label{products}
\ee
\be
r (A \circ B) \le r (A  B \circ B  A)^{1/2}\le  r (A  B),
\label{maxmixmax}
\ee
\be
r (A  B \circ B A) \le r (A ^2  B ^2),
\label{kvmax}
\ee

\label{Hadamard}
\end{corollary}

\begin{proof} 
 
By Theorem \ref{DPmax} and (\ref{commute})
 we have
$$r (A _1  \circ \cdots \circ A _m)^m = r\left( (A _1  \circ \cdots \circ A _m)^m  \right)$$
$$=  r\left( (A _1  \circ \cdots \circ A _m)  (A _2  \circ \cdots \circ A _m \circ A_1) \cdots (A _m  \circ A_1 \cdots \circ A _{m-1})\right)$$ 
$$\le r (P_1  \circ \cdots \circ P_m) \le  r (P_1)  \cdots  r (P_m) =  r(A_1  \cdots  A_m)^m , $$
which proves  (\ref{products}).

Inequality (\ref{maxmixmax}) is a special case of (\ref{products}), while (\ref{kvmax}) follows from (\ref{maxmixmax}) and (\ref{commute}).
\end{proof}

\begin{remark}
{\rm
 As pointed out in \cite[Example 4.10]{P12} 
the inequalities in  (\ref{products}) 
are sharp and may be strict, and in some cases the inequality (\ref{kvmax}) 
may be better than (\ref{maxmixmax}).
}
\end{remark}

If $m\in \NN$ and $q\in\PP_+$, $q=\sum_{j=0}^{\deg q}\al_jz^j$ let us define the polynomial $q^{[m]}$ by
$q^{[m]}=\sum_{j=0}^{\deg q}\al_j ^m z^j$.

By applying   (\ref{eq_r_q}) 
 and an idea from \cite{DP16}, 
we extend Corollary \ref{Hadamard} 
in the following way.

\begin{theorem}
Let $q\in\PP_+$, $q=\sum_{j=0}^{\deg q}\al_jz^j$ and $A_1, \ldots, A_m, A, B \in \mathcal{C}$. If $P_i$ for $i=1, \ldots ,m$ are as in Corollary \ref{max_ineq}, 
then the following inequalities hold:

\be
 r (q_\vee (A _1  \circ \cdots \circ A _m)) \le r (q_\vee ^{[m]}(P_1  \circ \cdots \circ P_m))^{1/m} \le r (q_\vee (A_1   \cdots  A_m)), 
\label{Pji}
\ee
\be
r (q_\vee (A \circ B)) \le r (q_\vee ^{[2]}(A B \circ B A))^{1/2}\le  r (q_\vee (A B)),
\label{m=2}
\ee

\be
r (q_\vee (A B \circ B A)) \le r (q_\vee (A ^2  B ^2 )).
\label{another}
\ee
\label{powerineq}
\end{theorem}

\begin{proof}

Since $q_{\vee}(\sqrt[m]{t}) = \sqrt[m]{q^{[m]} _{\vee} (t)}$ for $t \ge 0$, it follows from (\ref{eq_r_q}) and (\ref{products}) 
 that
$$r (q_\vee (A _1  \circ \cdots \circ A _m))= q_\vee (r (A _1  \circ \cdots \circ A _m)) $$
$$\le q_\vee \left(r (P_1  \circ \cdots \circ P_m)^{1/m} \right)= q_\vee ^{[m]}(r (P_1  \circ \cdots \circ P_m))^{1/m}$$
$$=  r (q_\vee ^{[m]}(P_1  \circ \cdots \circ P_m))^{1/m} \le  q_\vee ^{[m]}(r  (A_1   \cdots  A_m) ^m)^{1/m} $$
$$= q_\vee ( r (A_1   \cdots  A_m))= r (q_\vee (A_1   \cdots  A_m)),$$
which proves  (\ref{Pji}).

Inequality (\ref{m=2}) 
is a special case of 
(\ref{Pji}) 
and inequality (\ref{another}) 
is proved in a similar way as  (\ref{Pji}).

\end{proof}

\section*{Acknowledgments} The first author was supported by grants  GA CR 17-00941S and by RVO 67985840.

The second author acknowledges a partial support of  the Slovenian Research Agency (grants P1-0222 and J1-8133).


\medskip
\medskip

\end{document}